\documentclass{amsproc}

\newtheorem{theorem}{Theorem}[section]

\theoremstyle{definition}
\newtheorem{definition}[theorem]{Definition}
\newtheorem{example}[theorem]{Example}

\theoremstyle{remark}
\newtheorem{remark}[theorem]{Remark}

\newtheorem{proposition}[theorem]{Proposition}
\newtheorem{corollary}[theorem]{Corollary}

\catcode`\@=\active 

\numberwithin{equation}{section}

\begin{document}

\title{On the Characteristic Foliations of Metric Contact Pairs}


\author{Gianluca~Bande}
\address{Dipartimento di Matematica e Informatica, Universit{\`a} degli studi di Cagliari, Via Ospedale 72, 09124 Cagliari, Italy}
\email{gbande{\char'100}unica.it}

\author{Amine~Hadjar}
\address{Laboratoire de Math{\'e}matiques, Informatique et
Applications, Universit{\'e} de Haute Alsace, 4 Rue de
Fr{\`e}res Lumi{\`e}re, 68093 Mulhouse Cedex, France}
\email{mohamed.hadjar{\char'100}uha.fr}
\thanks{The first author was supported by the Project \textit{Start-up giovani ricercatori}--$2009$ of Universit\`a degli Studi di Cagliari. The second author was supported by a Visiting Professor fellowship at Universit\`a degli Studi di Cagliari in 2010,
financed by Regione Autonoma della Sardegna.}

\date{February 26, 2010 and, in revised form, February 26, 2010.}

\dedicatory{To John~C.~Wood for his sixtieth birthday}


\begin{abstract}
A contact pair on a manifold always admits an associated metric for which the two characteristic contact foliations are orthogonal.
We show that all these metrics have the same volume element. We also prove that the leaves of the characteristic foliations are minimal with respect to these metrics. We give an example where these leaves are not totally geodesic submanifolds.
\end{abstract}

\maketitle

\section{Introduction}
In a previous paper \cite{BH2} we have considered \textit{contact pair structures} and studied some properties of their associated metrics.
This notion was first introduced by Blair, Ludden and Yano \cite{Blair2} by the name {\it bicontact} structures, and is a special type of $f$-{\it structure} in the sense of Yano \cite{yano}.
More precisely, a \emph{metric contact pair} on an even dimensional manifold
is a triple $(\alpha_1 , \alpha_2 , g)$, where $(\alpha_1 , \alpha_2)$ is a contact pair (see \cite{BH}) with Reeb vector fields $Z_1$, $Z_2$, and $g$ a Riemannian metric such that $g(X, Z_i)=\alpha_i(X)$, for $i=1,2$, and for which the endomorphism field $\phi$ uniquely defined by $g(X, \phi Y)= (d \alpha_1 + d
\alpha_2) (X,Y)$ verifies
\begin{equation*}\label{d:cpstructure}
\phi^2=-Id + \alpha_1 \otimes Z_1 + \alpha_2 \otimes Z_2 , \;
\phi(Z_1)=\phi(Z_2)=0 .
\end{equation*}

Contact pairs always admit associated metrics with {\it decomposable} structure tensor $\phi$, i.e. $\phi$ preserves the two characteristic distributions of the pair (see \cite{BH2}).
In this paper, we first show that for a given contact pair all such associated metrics have the same volume element. Next we prove that with respect to these metrics the two characteristic foliations are orthogonal and minimal.
We end by giving an example where the leaves of the characteristic foliations are not totally geodesic.

All the differential objects considered in this paper are assumed to be smooth.

\section{Preliminaries on metric contact pairs}\label{s:prelim}
In this section we gather the notions concerning contact pairs
that will be needed in the sequel. We refer the reader to
\cite{Bande1, Bande2, BH, BH2, BH3, BGK, BK} for further informations and several
examples of such structures.

\subsection{Contact pairs and their characteristic foliations}\label{s:prelimcp}
Recall that a pair $(\alpha_1, \alpha_2)$ of $1$-forms on a manifold is said
to be a \emph{contact pair} of type $(h,k)$ if:
\begin{eqnarray*}
&\alpha_1\wedge (d\alpha_1)^{h}\wedge\alpha_2\wedge
(d\alpha_2)^{k} \;\text{is a volume form},\\
&(d\alpha_1)^{h+1}=0 \; \text{and} \;(d\alpha_2)^{k+1}=0.
\end{eqnarray*}
Since the form $\alpha_1$ (resp. $\alpha_2$) has constant class
$2h+1$ (resp. $2k+1$), the characteristic distribution $\ker \alpha_1 \cap \ker d\alpha_1$ (resp.
 $\ker \alpha_2 \cap \ker d\alpha_2$) is completely integrable and determines the so-called \emph{characteristic
  foliation} $\mathcal{F}_1$ (resp. $\mathcal{F}_2$) whose leaves are endowed with a contact form induced by $\alpha_2$ (resp. $\alpha_1$).

The equations
\begin{eqnarray*}
&\alpha_1 (Z_1)=\alpha_2 (Z_2)=1  , \; \; \alpha_1 (Z_2)=\alpha_2
(Z_1)=0 \, , \\
&i_{Z_1} d\alpha_1 =i_{Z_1} d\alpha_2 =i_{Z_2}d\alpha_1=i_{Z_2}
d\alpha_2=0 \, ,
\end{eqnarray*}
where $i_X$ is the contraction with the vector field $X$, determine completely the two vector fields $Z_1$ and $Z_2$, called \textit{Reeb vector fields}. 
Notice that $Z_i$ is nothing but the Reeb vector field of the contact form $\alpha_i$ on each leaf of $\mathcal{F}_j$ for $i\neq j$.

The tangent bundle of a manifold $M$ endowed with a contact pair
can be split in different ways. For $i=1,2$, let $T\mathcal F _i$
be the subbundle of $TM$ determined by the characteristic foliation of
$\alpha_i$, $T\mathcal G_i$ the subbundle whose fibers are
given by $\ker d\alpha_i \cap \ker \alpha_1 \cap \ker \alpha_2$
and $\mathbb{R} Z_1, \mathbb{R} Z_2$ the line bundles determined
by the Reeb vector fields. Then we have the following splittings:
\begin{equation*}
TM=T\mathcal F _1 \oplus T\mathcal F _2 =T\mathcal G_1 \oplus
T\mathcal G_2 \oplus \mathbb{R} Z_1 \oplus \mathbb{R} Z_2
\end{equation*}
Moreover we have $T\mathcal F _1=T\mathcal G_1 \oplus \mathbb{R}
Z_2 $ and $T\mathcal F _2=T\mathcal G_2 \oplus \mathbb{R} Z_1 $.

Notice that $d\alpha_1$ (resp. $d\alpha_2$) is symplectic on the vector bundle $T\mathcal G_2$ (resp. $T\mathcal G_1$).
\begin{example}
Take $(\mathbb{R}^{2h+2k+2},\alpha_1 , \alpha_2)$ where $\alpha_1$ (resp. $\alpha_2$) is the Darboux contact form on $\mathbb{R}^{2h+1}$
(resp. on $\mathbb{R}^{2k+1}$).
\end{example}
This is also a local model for all contact pairs of type $(h,k)$ (see \cite{Bande1, BH}). Hence a contact pair manifold is locally a product of two contact manifolds.

\subsection{Contact pair structures}
We recall now the definition of contact pair structure introduced in
\cite{BH2} and some basic properties.
\begin{definition}
A \emph{contact pair structure} on a manifold $M$ is a triple
$(\alpha_1 , \alpha_2 , \phi)$, where $(\alpha_1 , \alpha_2)$ is a
contact pair and $\phi$ a tensor field of type $(1,1)$ such that:
\begin{equation}\label{d:cpstructure}
\phi^2=-Id + \alpha_1 \otimes Z_1 + \alpha_2 \otimes Z_2 , \;
\phi(Z_1)=\phi(Z_2)=0
\end{equation}
where $Z_1$ and $Z_2$ are the Reeb vector fields of $(\alpha_1 ,
\alpha_2)$.
\end{definition}
It is easy to check that $\alpha_i \circ \phi =0$ for $i=1,2$, that the rank of $\phi$ is
equal to $\dim M -2$ , and that $\phi$ is almost complex on the vector bundle $T\mathcal G_1 \oplus
T\mathcal G_2$ .

Since we are also interested on the induced structures, we recall that
the endomorphism $\phi$ is said to be \textit{decomposable} if
$\phi (T\mathcal{F}_i) \subset T\mathcal{F}_i$, for $i=1,2$.
This condition is equivalent to $\phi(T\mathcal{G}_i)= T\mathcal{G}_i$.
In this case $(\alpha_1 , Z_1 ,\phi)$ (resp.
$(\alpha_2 , Z_2 ,\phi)$) induces, on every leaf of $\mathcal{F}_2$ (resp. $\mathcal{F}_1$), a contact form with structure tensor the restriction
of $\phi$ to the leaf.

\subsection{Metric contact pairs}
On manifolds endowed with contact pair structures it is natural
to consider the following kind of metrics:
\begin{definition}[\cite{BH2}]
Let $(\alpha_1 , \alpha_2 ,\phi )$ be a contact pair structure on
a ma-nifold $M$, with Reeb vector fields $Z_1$ and $Z_2$. A
Riemannian metric $g$ on $M$ is called:
\begin{enumerate}
\item \emph{compatible} if $g(\phi X,\phi Y)=g(X,Y)-\alpha_1 (X)
\alpha_1 (Y)-\alpha_2 (X) \alpha_2 (Y)$ for all vector fields $X$ and $Y$,
\item \emph{associated} if $g(X, \phi Y)= (d \alpha_1 + d
\alpha_2) (X,Y)$ and $g(X, Z_i)=\alpha_i(X)$, for $i=1,2$ and for
all vector fields $X,Y$. \label{ass-metric}
\end{enumerate}
\end{definition}
An associated metric is compatible, but the converse is not true.
\begin{definition}[\cite{BH2}]
A \emph{metric contact pair} (MCP) on a manifold $M$ is a
quadruple $(\alpha_1, \alpha_2, \phi, g)$ where $(\alpha_1,
\alpha_2, \phi)$ is a contact pair structure and $g$ an associated
metric with respect to it. The manifold $M$ is called a MCP manifold.
\end{definition}
Note that the equation
\begin{equation}\label{d:phi}
 g(X, \phi Y)= (d \alpha_1 + d\alpha_2) (X,Y)
\end{equation}
determines completely the endomorphism $\phi$.
So we can talk about a metric $g$ \emph{associated to a contact pair} $(\alpha_1 , \alpha_2)$ when  $g(X, Z_i)=\alpha_i(X)$, for $i=1,2$, and the endomorphism $\phi$ defined by equation  \eqref{d:phi} verifies \eqref{d:cpstructure}.

\begin{theorem}[\cite{BH2}]
For a MCP $(\alpha_1 , \alpha_2, \phi, g)$, the tensor $\phi$ is
decomposable if and only if the characteristic foliations $\mathcal{F}_1 ,
\mathcal{F}_2$ are orthogonal.
\end{theorem}

Using a standard polarization on the symplectic vector bundles $T\mathcal G_i$ (see Section \ref{s:prelimcp}), one can see that for a given contact pair $(\alpha_1, \alpha_2)$ there always exist a decomposable $\phi$ and a metric $g$ such
that $(\alpha_1, \alpha_2, \phi, g)$ is a MCP (see \cite {BH2}).
This can be stated as:

\begin{theorem}[\cite{BH2}]
For a given contact pair on a manifold,
there always exists an associated metric for which the characteristic foliations are orthogonal.
\end{theorem}

Let $(\alpha_1, \alpha_2 ,\phi, g )$ be a MCP on a manifold with decomposable $\phi$.
Then $(\alpha_i, \phi , g)$ induces a contact metric structure on the
leaves of the characteristic foliation $\mathcal{F}_j$ of $\alpha_j$, for $i \neq
j$ (see \cite{BH2}).

\begin{example}\label{mcp-product}
As a trivial example one can take two metric contact manifolds $(M_i, \alpha_i,g_i)$ and consider the MCP $(\alpha_1,\alpha_2,g_1 \oplus g_2)$ on $M_1\times M_2$. The characteristic foliations are given by the two trivial fibrations.
\end{example}

\begin{remark}
To get more examples of MCP on closed manifolds, one can imitate the constructions on flat bundles and Boothby-Wang fibrations given in \cite{BH3}
and adapt suitable metrics on the bases and fibers of these fibrations. See also Example \ref{liegroup} below which concerns a nilpotent Lie group and its closed nilmanifolds.
\end{remark}

\section{Minimal foliations}
Given any compatible metric $g$ on a manifold endowed with a contact pair structure $(\alpha_1, \alpha_2 , \phi)$ of type $(h,k)$,
with Reeb vector fields
$Z_1$ and $Z_2$, one can construct a local basis, called $\phi$-basis.
On an open set, on the orthogonal complement of $Z_1$ and $Z_2$, choose a vector field $X_1$ of
length $1$ and take $\phi X_1$. Then take the orthogonal
complement of $\{Z_1, Z_2, X_1 ,\phi X_1\}$ and so on. By iteration of this procedure,
one obtains a local orthonormal basis
$$
\{ Z_1 , Z_2 , X_1 , \phi X_1 , \cdots , X_{h+k} , \phi X_{h+k}\},
$$
which will be called $\phi$-basis and is the analog of a $\phi$-basis for almost contact structures, or $J$-basis in the case of
an almost complex structure $J$.

If $\phi$ is decomposable and  $g$ is an associated metric, since the characteristic foliations are orthogonal, it is possible to construct the $\phi$-basis in a better way. Starting with $X_1$ tangent to one of the characteristic foliations,
which are orthogonal, with a slight modification of the above
construction, we obtain a $\phi$-basis
$$
\{ Z_1 , X_1 , \phi X_1 , \cdots , X_h , \phi X_h ,Z_2 , Y_1 , \phi Y_1 , \cdots , Y_k ,
\phi Y_k  \}
$$
such that $\{ Z_1 , X_1 , \phi X_1 , \cdots , X_h ,
\phi X_h \}$ is a $\phi$-basis for the induced metric contact structures on the
leaves of $\mathcal F_2$, and
$\{Z_2 , Y_1 , \phi Y_1 , \cdots ,
Y_k , \phi Y_k  \}$ is a $\phi$-basis for the leaves of $\mathcal
F_1 $.

Using this basis and the formula for the volume form on
contact metric manifolds (see \cite{Blairbook}, for example), one can easily show the following:

\begin{proposition}\label{volumeform}
On a manifold endowed with a MCP $(\alpha_1,
\alpha_2, \phi, g)$ of type $(h,k)$, with a decomposable $\phi$,
the volume element of the Riemannian metric $g$ is given by:
\begin{equation}\label{MCP-volume}
dV= \frac{(-1)^{h+k}}{2^{h+k} h! k!} \alpha_1 \wedge (d\alpha_1) ^h \wedge \alpha_2 \wedge (d\alpha_2) ^k
\end{equation}
\end{proposition}

A direct application of the minimality criterion of Rummler (see \cite{Rum} page 227)
to the volume form on a
MCP manifold yields the following result:
\begin{theorem}
On a MCP manifold $(M, \alpha_1,\alpha_2, \phi,
g)$ with decomposable $\phi$, the characteristic foliations are
minimal.
\end{theorem}
\begin{proof}
Recall the minimality criterion of Rummler: let $\mathcal F$ be a $p$-dimensional foliation on a Riemannian manifold and $\omega$ its characteristic form (i.e. the $p$-form
which vanishes on vectors orthogonal to $\mathcal F$ and whose restriction to $\mathcal F$ is the volume of the induced metric on the leaves). Then
$\mathcal F$ is minimal iff $\omega$ is closed on $T\mathcal F$ (i.e. $d\omega(X_1,...,X_p,Y)=0$ for $X_1$, ..., $X_p$ tangent to $\mathcal F$).

Let $\mathcal F_i$ be the characteristic foliation of $\alpha_i$.
As the volume element of the Riemannian metric $g$ is given by \eqref{MCP-volume}, the characteristic form of $\mathcal F_1$ (resp. $\mathcal F_2$) is,
up to a constant, $\alpha_2 \wedge (d\alpha_2) ^k$ (resp. $\alpha_1 \wedge (d\alpha_1) ^h$). But these forms are closed
by  the contact pair condition, and then the
criterion applies directly.
\end{proof}

Since every manifold endowed with a contact pair always admits
an associated metric with decomposable $\phi$, we
have a statement already proved in \cite{BK}:
\begin{corollary}
On every manifold endowed with a contact pair there exists a
metric for which the characteristic foliations are orthogonal and minimal.
\end{corollary}

\begin{remark}
Although a contact pair manifold is locally a product of two contact manifolds (see Section \ref{s:prelimcp}), an associated metric for which the characteristic foliations are orthogonal is not necessary locally a product as in Example \ref{mcp-product}. Here is an interesting case:
\end{remark}
\begin{example}\label{liegroup}
Let us consider the simply connected $6$-dimensional nilpotent Lie group $G$ with structure
equations:
\begin{eqnarray*}
&d\omega_3= d\omega_6=0 \; \; , \; \; d\omega_2= \omega_5 \wedge
\omega _6 ,\\
&d\omega_1=\omega_3 \wedge \omega_4 \; \; , \; \;
d\omega_4= \omega_3 \wedge \omega_5 \; \; , \; \; d\omega_5 =
\omega_3 \wedge \omega_6 \, ,
\end{eqnarray*}
where the $\omega_i$'s form a basis for the cotangent space of $G$
at the identity.

The pair $(\omega_1 , \omega_2)$ is a contact pair of type $(1,1)$
with Reeb vector fields $(X_1, X_2)$, the $X_i$'s being dual to
the $\omega_i$'s. The characteristic distribution of $\omega_1$ (resp. $\omega_2$) is spanned by $X_2$, $X_5$ and $X_6$ (resp. $X_1$, $X_3$ and $X_4$).

The left invariant metric
\begin{equation}
g=\omega_1 ^2+\omega_2 ^2+\frac{1}{2}\sum_{i=3}^6 \omega_i ^2
\end{equation}
is associated to the contact pair $(\omega_1 , \omega_2)$ with decomposable structure tensor $\phi$ given by
$\phi(X_6)=X_5$ and $\phi (X_4)=X_3$ .

The characteristic foliations have minimal leaves. Moreover the leaves tangent to the identity of $G$ are Lie subgroups isomorphic to the Heisenberg group.

Notice that these foliations are not totally geodesic since $g(\nabla _{X_4} X_3,X_5)\neq0$ and $g(\nabla _{X_5} X_6,X_3)\neq0$, where $\nabla$ is the Levi-Civita connection of this metric.
So the metric $g$ is not locally a product.

Since the structure constants of the group are rational, there exist lattices $\Gamma$ such that $G/\Gamma$ is
compact. Since the MCP on G is left invariant, it descends to all quotients $G/\Gamma$ and we obtain closed
nilmanifolds carrying the same type of structure.
\end{example}

\bibliographystyle{amsalpha}

\end{document}